\documentclass[11pt,reqno]{amsart}
\usepackage{amsmath}
\usepackage{amssymb}
\usepackage{enumitem} 
\usepackage{amsfonts}
\usepackage{comment}

\usepackage{xcolor}
\usepackage{float}
\usepackage{mathtools}
\usepackage[colorlinks=true,citecolor=cyan, 
urlcolor=blue,linkcolor=blue]{hyperref}

\definecolor{fgreen}{RGB}{44,144, 14}

\renewenvironment{proof}{{\bfseries Proof.}}{\qed}

\usepackage[a4paper,margin=1in]{geometry}

\newtheorem{theorem}{Theorem}[section] 
\newtheorem{proposition}[theorem]{Proposition} 
 
\newtheorem{lemma}[theorem]{Lemma} 

\theoremstyle{definition}
 
\newtheorem{remark}[theorem]{Remark}

\everymath{\displaystyle}

\def\R{\mathbb R}

\def\C{\mathbb C}

\def\H{\mathbb H}

\def\ib{\mathbf {i}}
\def\jb{\mathbf {j}}
\def\kb{\mathbf {k}}

\def\R{\mathbb R}

\newcommand{\SL}{\mathrm{SL}}

\newcommand{\GL}{\mathrm{GL}}

\def\R{\mathbb {R}}
\def\C{\mathbb {C}}

\def\H{\mathbb {H}}

\def\ib{\mathbf {i}}

\def\jb{\mathbf {j}}

\def\GL{\rm GL}
\def\SL{\rm SL}

\newcommand{\PSL}{\mathrm{PSL}}

\makeatletter
\@namedef{subjclassname@2010}{\textup{2010} Mathematics Subject Classification}
\makeatother

\begin{document}

	\title[Reversibility of quaternionic Möbius transformations]
	{Algebraic characterization of reversibility in the quaternionic M\"obius group
	}
	
	\author[K. Gongopadhyay, T.  Lohan and  A. Mukherjee]{Krishnendu Gongopadhyay, Tejbir Lohan and Abhishek Mukherjee}

	\address{Indian Institute of Science Education and Research (IISER) Mohali,
		Knowledge City,  Sector  81, S.A.S. Nagar 140306, Punjab, India}	\email{krishnendu@iisermohali.ac.in \\ 
		\textit{ORCID}: \href{https://orcid.org/0000-0003-4327-0660}{0000-0003-4327-0660}}

	\address{Theoretical Statistics and Mathematics Unit,
		Indian Statistical Institute, Delhi Centre, New Delhi 110016, India}
	\email{tejbirlohan70@gmail.com	\\
		\textit{ORCID}: \href{https://orcid.org/0000-0002-6912-6221}{0000-0002-6912-6221}}

	\address{Kalna College, Kalna, Dist. Burdwan, West Bengal 713409, India}
	\email{abhishekmukherjee@kalnacollege.ac.in \\ 
		\textit{ORCID}: \href{https://orcid.org/0009-0000-9610-8398}{0009-0000-9610-8398}}



	\date{\today}
	%

	
	\begin{abstract}
		An element of a group is called \emph{reversible} if it is conjugate to its inverse.
		While reversibility in the quaternionic M\"{o}bius group $\mathrm{PSL}(2,\mathbb{H})$ 
		has traditionally been studied using geometric and dynamical methods, we develop a 
		purely algebraic approach.
		We obtain an explicit, computable criterion for the reversibility of a quaternionic 
		M\"{o}bius transformation, expressed solely in terms of the entries of a matrix representative.
		More precisely, we prove that
		\[
		[A]\in \mathrm{PSL}(2,\mathbb{H}) \text{ is reversible}
		\quad \Longleftrightarrow \quad
		\beta_A^{2}=\delta_A^{2},
		\]
		where $\beta_A$ and $\delta_A$ are real conjugacy invariants associated with a lift 
		$A\in \mathrm{SL}(2,\mathbb{H})$.
		Furthermore, we give a complete characterization of reversing symmetries 
		of reversible elements in $\mathrm{SL}(2,\mathbb{H})$ and $\mathrm{PSL}(2,\mathbb{H})$.
	\end{abstract}

	\maketitle
	
	\section{Introduction}\label{sec-intro}
	
	Reversibility concerns the study of symmetries whose behavior remains invariant under time reversal. 
	The concept can be traced back to the classical works of Birkhoff, Arnol'd, Devaney, and others; see \cite{Bi,AA,De,LR,BR}. 
	It has significant implications across diverse areas of mathematics, including classical dynamics, group theory, geometry, number theory, complex analysis, and representation theory; see \cite{O'Fa} for further details.
	
	An element of a group is called \emph{reversible} if it is conjugate to its inverse. 
	Reversible elements, also known as \emph{real elements}, arise naturally in various contexts in mathematics; see~\cite{OS, Sa, ES, dg, HMT}. 
	They are closely related to \emph{strongly reversible} or \emph{strongly real} elements, which can be expressed as a product of two involutions (elements of order at most two).
	Equivalently, an element is strongly reversible if and only if it is conjugate to its inverse by an involution. 
	The investigation of reversible and strongly reversible elements is a historically rich and active area of research; see the monograph~\cite{OS} for a comprehensive exposition.

	Let $\mathbb{H}$ denote the division algebra of Hamilton's quaternions. 
	For $a\in \mathbb{H}$, let $\mathfrak{R}(a)$ denote the real part of $a$. 
	Let $\mathrm{M}(2,\mathbb{H})$ be the algebra of $2\times 2$ matrices over $\mathbb{H}$. 
	For each matrix in $\mathrm{M}(2,\mathbb{H})$, the quaternionic determinant is defined in~\eqref{eq-def-det-char}. 
	Let $\GL(2,\mathbb{H})$ be the group of invertible matrices in $\mathrm{M}(2,\mathbb{H})$, and let $\SL(2,\mathbb{H})$ be the subgroup of $\GL(2,\mathbb{H})$ consisting of matrices with quaternionic determinant equal to~$1$.
	
	To understand the conformal geometry of the $n$-dimensional sphere, one may realize its conformal automorphisms as linear fractional transformations of $2\times 2$ Vahlen matrices over Clifford numbers; see Ahlfors~\cite{Ah} and Waterman~\cite{Wa}. 
	In the special case $n=4$, the four-dimensional sphere $\mathbb{S}^{4}$ can be identified with the extended quaternionic plane $\widehat{\mathbb{H}}:=\mathbb{H}\cup\{\infty\}$. 
	The group $\SL(2,\mathbb{H})$ acts on $\widehat{\mathbb{H}}$ by conformal diffeomorphisms via quaternionic linear fractional transformations,
	\[
	\begin{pmatrix} a & b \\ c & d \end{pmatrix} : z \longmapsto (az+b)(cz+d)^{-1}.
	\]
	This action extends to the orientation-preserving isometries of the real hyperbolic $5$-space $\mathbf{H}^{5}$, whose boundary at infinity is identified with $\widehat{\mathbb{H}}$. 
	The isometry group is naturally identified with $\PSL(2,\mathbb{H})$, defined as the quotient $\SL(2,\mathbb{H})/\{\pm I_{2}\}$. 
	For $A\in \SL(2,\mathbb{H})$, we denote its image in $\PSL(2,\mathbb{H})$ by $[A]$. 
	Each element $[A]\in \PSL(2,\mathbb{H})$ has precisely two lifts, $A$ and $-A$, in $\SL(2,\mathbb{H})$, and we often identify one of these lifts with the quaternionic M\"obius transformation corresponding to $[A]$. 
	For more details on this quaternionic formalism; see Wilker~\cite{Wi} and Parker--Short~\cite{PS}.
	
	In~\cite{LFS}, L\'avi\v{c}ka et al.\ investigate reversibility in $\PSL(2,\mathbb{H})$ using the dynamics of quaternionic M\"obius transformations. 
	They classify the reversible elements in $\PSL(2,\mathbb{H})$ and prove that every reversible element in $\PSL(2,\mathbb{H})$ is strongly reversible. 
	The classification of reversible elements in $\PSL(2,\mathbb{H})$ is also implicit in the works of Short~\cite{Sh} and Gongopadhyay~\cite{KG2}, where reversible isometries of $n$-dimensional hyperbolic space are classified in the identity component of the isometry group.
	
	Despite the extensive literature on the dynamics and geometry of quaternionic 
	M\"{o}bius transformations, to the best of our knowledge, no result provides an 
	explicit algebraic criterion for classifying reversible transformations.
	Motivated by this, we study reversibility in $\PSL(2,\mathbb{H})$ from a purely 
	algebraic perspective.
	Given a matrix $A \in \SL(2,\mathbb{H})$ representing a quaternionic M\"{o}bius 
	transformation, it is natural to ask how reversibility can be determined solely 
	from its matrix entries.
	While reversibility in $\PSL(2,\mathbb{H})$ has been studied dynamically and 
	geometrically in earlier work, and individual dynamical types of elements have 
	been characterized algebraically (see, for example,~\cite{Ca, Fo, KG, PS}), no 
	purely algebraic criterion for reversibility in terms of matrix invariants appears 
	in the literature.
	Our main contribution is precisely such a criterion, together with an explicit 
	algebraic description of reversing symmetries in both $\SL(2,\mathbb{H})$ and 
	$\PSL(2,\mathbb{H})$.

	We present an algebraic characterization of reversibility in $\PSL(2,\mathbb{H})$ and provide a simpler proof that reversible and strongly reversible elements coincide; see Proposition~\ref{prop-equiv-rev-str}. 
	Our approach characterizes reversibility using conjugacy invariants of matrix representatives in $\SL(2,\mathbb{H})$. 
	More precisely, for $[A]\in \PSL(2,\mathbb{H})$ and a lift $A\in \SL(2,\mathbb{H})$, the conjugacy invariants introduced in Section~\ref{sec-conj-invar} provide an explicit and computable criterion to determine whether $[A]$ is (strongly) reversible. 
	Our main result is the following.

	\begin{theorem}\label{thm-rev-alg-char}
		Let $[A]\in \PSL(2,\mathbb{H})$, and let
		\[
		A=\begin{pmatrix} a & b \\ c & d \end{pmatrix}\in \SL(2,\mathbb{H})
		\]
		be a lift of $[A]$. 
		Then $[A]$ is reversible in $\PSL(2,\mathbb{H})$ if and only if $\beta_A^{2}=\delta_A^{2}$, where $\beta_A$ and $\delta_A$ are the real conjugacy invariants of $A$ given by
		\[
		\beta_A=\mathfrak{R}\big((ad-bc)\bar a+(da-cb)\bar d\big)
		\quad\text{and}\quad
		\delta_A=\mathfrak{R}(a+d).
		\]
	\end{theorem}
	
	To prove Theorem~\ref{thm-rev-alg-char}, we first classify reversible elements in $\PSL(2,\mathbb{H})$ and then apply the conjugacy invariants of quaternionic matrices discussed in Section~\ref{sec-conj-invar}. 
	Our approach differs from those in~\cite{LFS,KG2,Sh}, but it is inspired by the ideas in~\cite{dgl2} for $\PSL(3,\mathbb{H})$, which relate reversibility to the classification of dynamical types of projective automorphisms of the quaternionic projective plane $\mathbb{P}_{\mathbb{H}}^{2}$.

	This approach also yields a clearer picture of reversibility in the linear group $\SL(2,\mathbb{H})$. 
	Unlike $\PSL(2,\mathbb{H})$, the group $\SL(2,\mathbb{H})$ contains reversible elements that are not strongly reversible; for example,
	$
	A=\begin{psmallmatrix}
		\mathbf{i} & 0\\
		0 & 1
	\end{psmallmatrix}\in \SL(2,\mathbb{H}).
	$
	This shows that classifying strongly reversible elements in $\SL(2,\mathbb{H})$ is a problem of independent interest. 
	We note that reversibility in $\SL(n,\mathbb{H})$ for arbitrary $n$ has been classified only recently in~\cite{GLM} using a more general approach involving special canonical forms. 
	In this paper we use explicit algebraic computations to classify reversible and strongly reversible elements in $\SL(2,\mathbb{H})$; see Lemma~\ref{lem-rev-SL(2,H)} and Lemma~\ref{lem-classification-str-rev-SL}. 
	These classifications are key ingredients in the proof of Theorem~\ref{thm-rev-alg-char}.
	
	Another outcome of our method is an explicit description of reversing symmetries. 
	Recall that a \emph{reversing symmetry} (also called a \emph{time-reversal symmetry} or a \emph{reverser}) of a reversible element $g$ in a group $G$ is an element of $G$ that conjugates $g$ to $g^{-1}$; see \cite{LR,BR}. 
	A complete description of reversing symmetries for reversible quaternionic M\"obius transformations does not appear to be available in the existing literature. 
	Using our approach, such descriptions follow from Lemma~\ref{lem-reverser} and Lemma~\ref{lem-reverser-PSL}.

	\noindent\textbf{Structure of the paper.}
	In Section~\ref{sec-prel}, we recall preliminary results on conjugacy invariants, conjugacy classes, and centralizers in $\SL(2,\mathbb{H})$. 
	Section~\ref{sec-rev} classifies reversible elements in $\SL(2,\mathbb{H})$ and describes their reversing symmetries, and then uses these results to study strongly reversible elements in $\SL(2,\mathbb{H})$. 
	This section also discusses decompositions of reversible elements in $\SL(2,\mathbb{H})$ as products of two skew-involutions. 
	In Section~\ref{sec-rev-projective-group}, we investigate elements $A\in \SL(2,\mathbb{H})$ that are conjugate to $-A^{-1}$ and complete the proof of Theorem~\ref{thm-rev-alg-char}. 
	In Section~\ref{sec-inv-skew}, we discuss reversing symmetries and involutions in $\PSL(2,\mathbb{H})$.

	\section{Preliminaries}\label{sec-prel}
	
	Let $\H:= \R + \R \ib + \R \jb + \R \kb$ be the division algebra of Hamilton's quaternions, where
	$\ib^2=\jb^2=\kb^2=\ib\jb\kb=-1$.
	For $a= a_0 + a_1 \ib + a_2 \jb + a_3 \kb \in \H$, the conjugate, real part, and imaginary part of $a$
	are defined by
	\[
	\bar{a}:= a_0 - a_1 \ib - a_2 \jb - a_3 \kb,\qquad 
	\mathfrak{R}(a):= a_0,\qquad 
	\mathfrak{I}(a):= a_1 \ib + a_2 \jb + a_3 \kb,
	\]
	respectively, where $a_m \in \R$ for all $0 \leq m \leq 3$.
	We identify the real subspace $\R + \R \ib$ with the complex plane $\C$, and hence we may write
	$\H=\C+\C\jb$.
	
	Consider the algebra $\mathrm{M}(2,\H)$ of $2\times 2$ matrices over $\H$.
	A nonzero vector $v\in \H^2$ is called a (right) eigenvector of $A\in \mathrm{M}(2,\H)$
	corresponding to a (right) eigenvalue $\lambda\in \H$ if $Av=v\lambda$.
	Eigenvalues of quaternionic matrices occur in similarity classes, and each such class contains a unique
	complex number with nonnegative imaginary part. 
	Unless otherwise stated, we regard these unique complex representatives as the eigenvalues of $A$;
	see \cite{rodman} for a detailed discussion of quaternionic linear algebra.
	
	Let $\Phi$ be the embedding $\mathrm{M}(2,\H)\to \mathrm{M}(4,\C)$ defined by
	\begin{equation}\label{eq-def-embedd-phi}
		\Phi(A):=
		\begin{pmatrix}
			A_1 & A_2\\
			-\overline{A_2} & \overline{A_1}
		\end{pmatrix},
	\end{equation}
	where $A_1,A_2\in \mathrm{M}(2,\C)$ are such that $A=A_1+A_2\jb$, and $\overline{A_i}$ denotes the
	complex conjugate of $A_i$.
	The quaternionic determinant ${\det}_{\H}(A)$ and the characteristic polynomial $\chi_{\H}(A)$ of
	$A$ are defined by
	\begin{equation}\label{eq-def-det-char}
		{\det}_{\H}(A):=\det(\Phi(A))
		\qquad\text{and}\qquad
		\chi_{\H}(A):=\chi_{\Phi(A)},
	\end{equation}
	respectively; see \cite[p.~113]{rodman}.
	By the Skolem--Noether theorem, these definitions are independent of the choice of the embedding $\Phi$.
	We define
	\begin{equation}\label{eq-def-GL-SL}
		\GL(2,\H):=\{A\in \mathrm{M}(2,\H)\mid {\det}_{\H}(A)\neq 0\},
		\qquad
		\SL(2,\H):=\{A\in \GL(2,\H)\mid {\det}_{\H}(A)=1\}.
	\end{equation}
	
	\subsection{Conjugacy classes in $\SL(2,\H)$}\label{subsec-conju}
	
	In the following lemma, we recall the well-known conjugacy classification in $\SL(2,\H)$, which follows
	from the Jordan decomposition of quaternionic matrices; see \cite[Theorem~5.5.3]{rodman}.
	
	\begin{lemma}\label{lem-conj}
		Every element of $\SL(2,\H)$ is conjugate to one of the following matrices:
		\begin{enumerate}[label={\normalfont(\roman*)}]
			\item 
			$\begin{pmatrix}
				r e^{\ib \theta} & 0\\
				0 & r^{-1} e^{\ib \phi}
			\end{pmatrix}$,
			where $r\in \R$, $r>0$, and $\theta,\phi\in [0,\pi]$.
			
			\item 
			$\begin{pmatrix}
				e^{\ib \theta} & 1\\
				0 & e^{\ib \theta}
			\end{pmatrix}$,
			where $\theta\in [0,\pi]$.
		\end{enumerate}
	\end{lemma}
	
	In view of Lemma~\ref{lem-conj}, we obtain the following dynamical classification of elements in
	$\PSL(2,\H)$ into three mutually exclusive types.
	
	\begin{lemma}\label{lem-class-name}
		Let $[A]\in \PSL(2,\H)$, where $A$ is a lift of $[A]$ in $\SL(2,\H)$. Then:
		\begin{enumerate}[label={\normalfont(\roman*)}]
			\item $[A]$ is \emph{hyperbolic} if $A$ has at least one eigenvalue of non-unit modulus.
			\item $[A]$ is \emph{elliptic} if $A$ is diagonalizable and every eigenvalue of $A$ has unit modulus.
			\item $[A]$ is \emph{parabolic} if $A$ is nondiagonalizable and every eigenvalue of $A$ has unit modulus.
		\end{enumerate}
	\end{lemma}
	
	\subsection{Conjugacy invariants of a quaternionic matrix}\label{sec-conj-invar}
	
	The following result is a consequence of \cite{Fo,KL} and provides conjugacy invariants for quaternionic
	matrices in $\mathrm{M}(2,\H)$.
	
	\begin{lemma}[cf.~{\cite[Section 4]{KL}}] \label{lem-conjugacy-invariant}
		Let $A=\begin{pmatrix} a & b\\ c & d \end{pmatrix}\in \mathrm{M}(2,\H)$. 
		Then the following functions on $\mathrm{M}(2,\H)$ are conjugation invariants:
		\begin{enumerate}[label={\normalfont(\roman*)}]
			\item $\alpha_A := |a|^2|d|^2 + |b|^2|c|^2 - 2\,\mathfrak{R}(a\bar{c}d\bar{b})$.
			\item $\beta_A := \mathfrak{R}\big((ad-bc)\bar{a}+(da-cb)\bar{d}\big)$.
			\item $\gamma_A := |a+d|^{2} + 2\,\mathfrak{R}(ad-bc)$.
			\item $\delta_A := \mathfrak{R}(a+d)$.
		\end{enumerate}
	\end{lemma}
	
	Note that $\beta_A=-\beta_{-A}$ and $\delta_A=-\delta_{-A}$ for all $A\in \mathrm{M}(2,\H)$.
	Since $A$ and $-A$ are the two possible lifts of an element $[A]\in \PSL(2,\H)$, the quantities
	$\beta_A$ and $\delta_A$ depend on the choice of lift.
	However, the invariants $\alpha_A$, $\beta_A^2$, $\gamma_A$, and $\delta_A^2$ are well-defined for
	$[A]\in \PSL(2,\H)$.
	In particular, the statement of Theorem~\ref{thm-rev-alg-char} is independent of the chosen lift.
	
	\subsection{Characteristic polynomial of a quaternionic matrix}\label{sec-char-poly}
	
	Suppose that the characteristic polynomial of
	$A=\begin{pmatrix} a & b\\ c & d \end{pmatrix}\in \mathrm{M}(2,\H)$ is
	\begin{equation}\label{eq-char-GL}
		\chi_{\H}(A)=x^4-c_3x^3+c_2x^2-c_1x+c_0,
	\end{equation}
	where $c_0={\det}_{\H}(A)$ is the quaternionic determinant of $A$.
	In view of \cite[Section~4]{KL}, we have
	\begin{equation}\label{eq-char-GL-relation}
		c_3=2\delta_A,\qquad
		c_2=\gamma_A,\qquad
		c_1=2\beta_A,\qquad
		c_0=\alpha_A.
	\end{equation}
	If $A\in \SL(2,\H)$, then $c_0=\alpha_A=1$. 
	Moreover, for $A\in \SL(2,\H)$, we have
	\begin{equation}\label{eq-char-inverse}
		\chi_{\H}(-A)=x^4+c_3x^3+c_2x^2+c_1x+1
		\quad\text{and}\quad
		\chi_{\H}(A^{-1})=x^4-c_1x^3+c_2x^2-c_3x+1.
	\end{equation}
	By Lemma~\ref{lem-conjugacy-invariant} and~\eqref{eq-char-GL-relation}, each coefficient $c_i\in \R$ is a
	conjugacy invariant of $A$, and can be expressed explicitly in terms of the entries $a,b,c,d$ of $A$.
	
	\subsection{Centralizers in $\SL(2,\H)$}
	
	For $a\in \H$, define the \emph{centralizer} of $a$ in $\H$ by
	\begin{equation}\label{eq-def-center-quat}
		\mathcal{Z}(a):=\{x\in \H\mid xa=ax\}.
	\end{equation}
	Using \cite[Lemma~2.6(2)]{dgl2}, for $\theta\in [0,\pi]$ we have
	\begin{equation}\label{eq-center-quat}
		\mathcal{Z}(e^{\ib \theta})=
		\begin{cases}
			\C, & \text{if } \theta\in (0,\pi),\\
			\H, & \text{if } \theta\in \{0,\pi\}.
		\end{cases}
	\end{equation}
	
	The \emph{centralizer} of an element $A\in \SL(2,\H)$ is
	\begin{equation}\label{eq-def-cent}
		\mathcal{Z}_{\SL(2, \, \H)}(A):=\{g\in \SL(2,\H)\mid gAg^{-1}=A\}.
	\end{equation}
	Centralizers in $\GL(2,\H)$ are well-studied; see \cite[Section~5]{KG}.
	The following lemma describes the centralizer of each conjugacy class representative in $\SL(2,\H)$.
	
	\begin{lemma}[\cite{KG}]\label{lem-centralizer}
		Let $\theta,\phi\in [0,\pi]$ and $r\in \R$, $r>0$. Then:
		\begin{enumerate}[label={\normalfont(\roman*)}]
			\item If $A=\begin{pmatrix} e^{\ib \theta} & 0\\ 0 & e^{\ib \theta} \end{pmatrix}$, then
			$
			\mathcal{Z}_{\SL(2, \,\H)}(A)=
			\left\{
			\begin{pmatrix} a & b\\ c & d \end{pmatrix}
			\ \middle|\ 
			a,b,c,d\in \mathcal{Z}(e^{\ib \theta})
			\right\}.
			$
			
			\item If $A=\begin{pmatrix} r e^{\ib \theta} & 0\\ 0 & r^{-1} e^{\ib \phi} \end{pmatrix}$ with
			$\theta\neq \phi$, then
			$
			\mathcal{Z}_{\SL(2, \,\H)}(A)=
			\left\{
			\begin{pmatrix} a & 0\\ 0 & b \end{pmatrix}
			\ \middle|\ 
			a\in \mathcal{Z}(e^{\ib \theta}),\ b\in \mathcal{Z}(e^{\ib \phi})
			\right\}.
			$
			
			\item If $A=\begin{pmatrix} r e^{\ib \theta} & 0\\ 0 & r^{-1} e^{\ib \theta} \end{pmatrix}$ with
			$r\neq 1$, then
			$
			\mathcal{Z}_{\SL(2, \,\H)}(A)=
			\left\{
			\begin{pmatrix} a & 0\\ 0 & b \end{pmatrix}
			\ \middle|\ 
			a,b\in \mathcal{Z}(e^{\ib \theta})
			\right\}.
			$
			
			\item If $A=\begin{pmatrix} e^{\ib \theta} & 1\\ 0 & e^{\ib \theta} \end{pmatrix}$, then
			$
			\mathcal{Z}_{\SL(2, \,\H)}(A)=
			\left\{
			\begin{pmatrix} a & b\\ 0 & a \end{pmatrix}
			\ \middle|\ 
			a,b\in \mathcal{Z}(e^{\ib \theta})
			\right\}.
			$
		\end{enumerate}
	\end{lemma}
	
	\subsection{Extended centralizers in $\SL(2,\H)$}\label{sec-rev-symm}
	
	A \emph{reversing symmetry} (also called a \emph{time-reversal symmetry} or a \emph{reverser}) of a
	reversible element $A$ is an element of $\SL(2,\H)$ that conjugates $A$ to $A^{-1}$; see \cite{LR,BR}.
	The set of all reversing symmetries of $A$ is
	\begin{equation}\label{eq-def-rev}
		\mathcal{R}_{\SL(2, \, \H)}(A):=\{g\in \SL(2,\H)\mid gAg^{-1}=A^{-1}\}.
	\end{equation}
	Note that $\mathcal{R}_{\SL(2, \, \H)}(A)$ is a right coset of the centralizer $\mathcal{Z}_{\SL(2, \, \H)}(A)$.
	Moreover, the \emph{extended centralizer} of $A$ is defined by
	\[
	\mathcal{E}_{\SL(2, \,\H)}(A):=\mathcal{Z}_{\SL(2, \, \H)}(A)\cup \mathcal{R}_{\SL(2, \, \H)}(A),
	\]
	and forms a subgroup of $\SL(2,\H)$ in which $\mathcal{Z}_{\SL(2, \, \H)}(A)$ has index at most $2$; see
	\cite{BR} and \cite[Section~2.1.4]{OS}.
	
	We provide a complete description of reversing symmetries of reversible elements in $\SL(2,\H)$ in
	Lemma~\ref{lem-reverser}; see Section~\ref{sec-inv-skew} for the corresponding result in $\PSL(2,\H)$.
	
	\section{Reversibility in $\mathrm{SL}(2,\mathbb{H})$}\label{sec-rev}
	
	In this section, we investigate the reversibility problem in the group $\mathrm{SL}(2,\mathbb{H})$.
	The following lemma classifies reversible elements in $\mathrm{SL}(2,\mathbb{H})$.
	
	\begin{lemma}\label{lem-rev-SL(2,H)}
		An element $A\in \mathrm{SL}(2,\mathbb{H})$ is reversible in $\mathrm{SL}(2,\mathbb{H})$ if and only if it is conjugate in $\mathrm{SL}(2,\mathbb{H})$ to one of the following matrices:
		\begin{enumerate}[label={\normalfont(\arabic*)}]
			\item\label{rev-type-1}
			$\begin{pmatrix}
				e^{\ib \theta} & 0\\
				0 & e^{\ib \theta}
			\end{pmatrix}$, where $\theta\in [0,\pi]$.
			
			\item\label{rev-type-2}
			$\begin{pmatrix}
				e^{\ib \theta} & 0\\
				0 & e^{\ib \phi}
			\end{pmatrix}$, where $\theta,\phi\in [0,\pi]$ and $\theta\neq \phi$.
			
			\item\label{rev-type-3}
			$\begin{pmatrix}
				r e^{\ib \theta} & 0\\
				0 & r^{-1} e^{\ib \theta}
			\end{pmatrix}$, where $r\in \R$, $r>0$, $r\neq 1$, and $\theta\in [0,\pi]$.
			
			\item\label{rev-type-4}
			$\begin{pmatrix}
				e^{\ib \theta} & 1\\
				0 & e^{\ib \theta}
			\end{pmatrix}$, where $\theta\in [0,\pi]$.
		\end{enumerate}
	\end{lemma}
	
	\begin{proof}
		For $A\in \mathrm{SL}(2,\mathbb{H})$ of one of the types~\ref{rev-type-1}--\ref{rev-type-4}, consider the following matrices $g\in \mathrm{SL}(2,\mathbb{H})$, respectively:
		\[
		\begin{array}{llll}
			\text{(i)}~
			g := \begin{pmatrix} \jb & 0 \\ 0 & \jb \end{pmatrix}, & \quad
			\text{(ii)}~
			g := \begin{pmatrix} \jb & 0 \\ 0 & \jb \end{pmatrix}, & \quad
			\text{(iii)}~
			g := \begin{pmatrix} 0 & \jb \\ \jb & 0 \end{pmatrix}, & \quad
			\text{(iv)}~
			g := \begin{pmatrix} -e^{-2\ib \theta}\,\jb & 0 \\ 0 & \jb \end{pmatrix}.
		\end{array}
		\]
		In each case, a straightforward calculation gives $gAg^{-1}=A^{-1}$.
		Hence $A$ is reversible in $\mathrm{SL}(2,\mathbb{H})$.
		
		Conversely, recall that $\jb z=\bar z\,\jb$ for all $z\in \C$.
		For a unique complex representative $\lambda\in \C$ of an eigenvalue class of $A$, the eigenvalues
		$\lambda$ and $\lambda^{-1}$ are conjugate if and only if $\lambda^{-1}=\overline{\lambda}$.
		Moreover, $A$ is reversible in $\mathrm{SL}(2,\mathbb{H})$ if and only if $A$ and $A^{-1}$ have the same Jordan form.
		Therefore, if $A$ is reversible, then by Lemma~\ref{lem-conj} it must be conjugate to one of the matrices listed in~\ref{rev-type-1}--\ref{rev-type-4}.
	\end{proof}
	
	\begin{remark}\label{rem-non-rev}
		In view of Lemmas~\ref{lem-conj} and~\ref{lem-rev-SL(2,H)}, an element $A\in \mathrm{SL}(2,\mathbb{H})$ is not reversible if and only if it is conjugate to
		\[
		\begin{pmatrix}
			r e^{\ib \theta} & 0\\
			0 & r^{-1} e^{\ib \phi}
		\end{pmatrix},
		\]
		where $r\in \R$, $r>0$, $\theta,\phi\in [0,\pi]$, and $r\neq 1$, $\theta\neq \phi$.
		In particular, $A\in \mathrm{SL}(2,\mathbb{H})$ is not reversible if and only if $\beta_A\neq \delta_A$.
		\qed
	\end{remark}
	
	\subsection{Strongly reversible elements in $\mathrm{SL}(2,\mathbb{H})$}\label{sec-str-rev}
	
	In this subsection, we classify strongly reversible elements in $\mathrm{SL}(2,\mathbb{H})$.
	We begin with the following lemma, which describes the set of reversing symmetries of a reversible element in $\mathrm{SL}(2,\mathbb{H})$.
	The notation $\mathcal{Z}(a)$ for $a\in \mathbb{H}$ is defined in~\eqref{eq-def-center-quat}.
	
	\begin{lemma}\label{lem-reverser}
		Let $A\in \mathrm{SL}(2,\mathbb{H})$ be a reversible element of one of the types~\ref{rev-type-1}--\ref{rev-type-4} in Lemma~\ref{lem-rev-SL(2,H)}.
		Then the set $\mathcal{R}_{\mathrm{SL}(2, \, \mathbb{H})}(A)$ of reversing symmetries of $A$ is given as follows, respectively:
		\begin{enumerate}[label={\normalfont(\roman*)}]
			\item
			$
			\mathcal{R}_{\mathrm{SL}(2, \, \mathbb{H})}(A)
			=
			\left\{
			\begin{pmatrix}
				a\,\jb & b\,\jb\\
				c\,\jb & d\,\jb
			\end{pmatrix}
			\ \middle|\
			a,b,c,d\in \mathcal{Z}(e^{\ib \theta})
			\right\}.
			$
			
			\item
			$
			\mathcal{R}_{\mathrm{SL}(2, \, \mathbb{H})}(A)
			=
			\left\{
			\begin{pmatrix}
				a\,\jb & 0\\
				0 & b\,\jb
			\end{pmatrix}
			\ \middle|\
			a\in \mathcal{Z}(e^{\ib \theta}),\ b\in \mathcal{Z}(e^{\ib \phi})
			\right\}.
			$
			
			\item
			$
			\mathcal{R}_{\mathrm{SL}(2, \, \mathbb{H})}(A)
			=
			\left\{
			\begin{pmatrix}
				0 & a\,\jb\\
				b\,\jb & 0
			\end{pmatrix}
			\ \middle|\
			a,b\in \mathcal{Z}(e^{\ib \theta})
			\right\}.
			$
			
			\item
			$
			\mathcal{R}_{\mathrm{SL}(2, \, \mathbb{H})}(A)
			=
			\left\{
			\begin{pmatrix}
				-(e^{-2\ib \theta}a)\,\jb & b\,\jb\\
				0 & a\,\jb
			\end{pmatrix}
			\ \middle|\
			a,b\in \mathcal{Z}(e^{\ib \theta})
			\right\}.
			$
		\end{enumerate}
	\end{lemma}
	
	\begin{proof}
		Recall that $\mathcal{R}_{\mathrm{SL}(2, \, \mathbb{H})}(A)$ is a right coset of the centralizer $\mathcal{Z}_{\mathrm{SL}(2, \, \mathbb{H})}(A)$; see Section~\ref{sec-rev-symm}.
		Thus, if $g\in \mathcal{R}_{\mathrm{SL}(2, \, \mathbb{H})}(A)$, then
		\[
		\mathcal{R}_{\mathrm{SL}(2, \, \mathbb{H})}(A)=\mathcal{Z}_{\mathrm{SL}(2, \, \mathbb{H})}(A)\,g.
		\]
		For each reversible element $A$ of types~\ref{rev-type-1}--\ref{rev-type-4}, choose $g\in \mathcal{R}_{\mathrm{SL}(2, \, \mathbb{H})}(A)$ as in the proof of Lemma~\ref{lem-rev-SL(2,H)}.
		The result now follows from Lemma~\ref{lem-centralizer}.
	\end{proof}
	
	We now use Lemma~\ref{lem-reverser} to study strongly reversible elements in $\mathrm{SL}(2,\mathbb{H})$.
	The next lemma identifies reversible elements in $\mathrm{SL}(2,\mathbb{H})$ that are not strongly reversible.
	
	\begin{lemma}\label{lem-non-strong-rev-SL(2,H)}
		Suppose that $A\in \mathrm{SL}(2,\mathbb{H})$ is one of the following matrices:
		\begin{enumerate}[label={\normalfont(\arabic*)}]
			\item\label{non-strong-rev-type-1}
			$\begin{pmatrix}
				e^{\ib \theta} & 0\\
				0 & e^{\ib \phi}
			\end{pmatrix}$, where $\theta\in (0,\pi)$ and $\phi\in [0,\pi]$ with $\theta\neq \phi$;
			
			\item\label{non-strong-rev-type-2}
			$\begin{pmatrix}
				e^{\ib \theta} & 1\\
				0 & e^{\ib \theta}
			\end{pmatrix}$, where $\theta\in (0,\pi)$.
		\end{enumerate}
		Then $A$ is not strongly reversible in $\mathrm{SL}(2,\mathbb{H})$.
	\end{lemma}
	
	\begin{proof}
		By Lemma~\ref{lem-rev-SL(2,H)}, the matrix $A$ is reversible in $\mathrm{SL}(2,\mathbb{H})$.
		Let $g\in \mathcal{R}_{\mathrm{SL}(2, \, \mathbb{H})}(A)$.
		Depending on whether $A$ is of type~\ref{non-strong-rev-type-1} or~\ref{non-strong-rev-type-2}, Lemma~\ref{lem-reverser} implies that $g$ must have one of the following forms, respectively:
		\begin{enumerate}[label={\normalfont(\roman*)}]
			\item
			$
			g=
			\begin{pmatrix}
				z\,\jb & 0\\
				0 & b
			\end{pmatrix},
			\quad
			\text{where } z\in \C \text{ and } b\in \mathcal{Z}(e^{\ib \phi});
			$
			
			\item
			$
			g=
			\begin{pmatrix}
				-(e^{-2\ib \theta}z)\,\jb & w\,\jb\\
				0 & z\,\jb
			\end{pmatrix},
			\quad
			\text{where } z,w\in \C.
			$
		\end{enumerate}
		For $z\in \C$, we have
		\begin{equation}\label{eq-1-non-strong-rev-SL(2,H)}
			(z\jb)^2=-|z|^2\neq 1.
		\end{equation}
		Hence $g^2\neq \mathrm{I}_2$, so there is no involution in $\mathcal{R}_{\mathrm{SL}(2, \, \mathbb{H})}(A)$.
		Therefore, $A$ is not strongly reversible in $\mathrm{SL}(2,\mathbb{H})$.
	\end{proof}
	
	The following lemma classifies strongly reversible elements in $\mathrm{SL}(2,\mathbb{H})$.
	
	\begin{lemma}\label{lem-classification-str-rev-SL}
		Let $A$ be a reversible element of $\mathrm{SL}(2,\mathbb{H})$.
		Then $A$ is strongly reversible in $\mathrm{SL}(2,\mathbb{H})$ if and only if it is conjugate to one of the following matrices:
		\begin{enumerate}[label={\normalfont(\arabic*)}]
			\item\label{strong-rev-type-1}
			$\begin{pmatrix}
				e^{\ib \theta} & 0\\
				0 & e^{\ib \theta}
			\end{pmatrix}$, where $\theta\in (0,\pi)$;
			
			\item\label{strong-rev-type-2}
			$\begin{pmatrix}
				e^{\ib \theta} & 0\\
				0 & e^{\ib \phi}
			\end{pmatrix}$, where $\theta,\phi\in \{0,\pi\}$;
			
			\item\label{strong-rev-type-3}
			$\begin{pmatrix}
				r e^{\ib \theta} & 0\\
				0 & r^{-1} e^{\ib \theta}
			\end{pmatrix}$, where $r\in \R$, $r>0$, $r\neq 1$, and $\theta\in [0,\pi]$;
			
			\item\label{strong-rev-type-4}
			$\begin{pmatrix}
				e^{\ib \theta} & 1\\
				0 & e^{\ib \theta}
			\end{pmatrix}$, where $\theta\in \{0,\pi\}$.
		\end{enumerate}
	\end{lemma}
	
	\begin{proof}
		By Lemmas~\ref{lem-rev-SL(2,H)} and~\ref{lem-non-strong-rev-SL(2,H)}, every strongly reversible element $A\in \mathrm{SL}(2,\mathbb{H})$ is conjugate to one of the matrices in~\ref{strong-rev-type-1}--\ref{strong-rev-type-4}.
		Conversely, for each matrix in~\ref{strong-rev-type-1}--\ref{strong-rev-type-4}, consider the following element $g\in \mathrm{SL}(2,\mathbb{H})$, respectively:
		\[
		\begin{array}{llll}
			\text{(i)}~ g := \begin{pmatrix} 0 & \jb \\ -\jb & 0 \end{pmatrix}, &
			\quad
			\text{(ii)}~ g := \begin{pmatrix} 1 & 0 \\ 0 & 1 \end{pmatrix}, &
			\quad
			\text{(iii)}~ g := \begin{pmatrix} 0 & \jb \\ -\jb & 0 \end{pmatrix}, &
			\quad
			\text{(iv)}~ g := \begin{pmatrix} 1 & 0 \\ 0 & -1 \end{pmatrix}.
		\end{array}
		\]
		In each case, $g$ is an involution in $\mathrm{SL}(2,\mathbb{H})$ and satisfies $gAg^{-1}=A^{-1}$.
		Hence $A$ is strongly reversible.
	\end{proof}
	
	Note that there exist reversible elements in $\mathrm{SL}(2,\mathbb{H})$ that are not strongly reversible and therefore cannot be written as a product of two involutions; see Lemma~\ref{lem-non-strong-rev-SL(2,H)}.
	However, every reversible element admits a factorization into skew-involutions.
	Recall that $g\in \mathrm{SL}(2,\mathbb{H})$ is called a \emph{skew-involution} if $g^2=-\mathrm{I}_2$.
	
	\begin{proposition}\label{prop-rev-prod-skew-inv}
		Every reversible element of $\mathrm{SL}(2,\mathbb{H})$ can be written as a product of two skew-involutions.
	\end{proposition}
	
	\begin{proof}
		Let $A\in \mathrm{SL}(2,\mathbb{H})$ be reversible.
		By Lemma~\ref{lem-rev-SL(2,H)}, we may assume that $A$ is one of the matrices listed in~\ref{rev-type-1}--\ref{rev-type-4}.
		Let $g$ be as chosen in the proof of Lemma~\ref{lem-rev-SL(2,H)}. Then $g$ is a skew-involution and satisfies $gAg^{-1}=A^{-1}$.
		Finally, note that $A$ can be written as a product of two skew-involutions if and only if there exists $g\in \mathrm{SL}(2,\mathbb{H})$ such that $gAg^{-1}=A^{-1}$ and $g^2=-\mathrm{I}_2$.
		This completes the proof.
	\end{proof}
	
	\section{Reversibility in $\mathrm{PSL}(2,\mathbb{H})$}\label{sec-rev-projective-group}
	
	In this section, we study reversibility in the projective linear group $\mathrm{PSL}(2,\mathbb{H})$. 
	Recall that each element $[A]\in \mathrm{PSL}(2,\mathbb{H})$ has precisely two lifts, $A$ and $-A$, in $\mathrm{SL}(2,\mathbb{H})$. 
	Therefore, $[A]\in \mathrm{PSL}(2,\mathbb{H})$ is reversible if and only if either $gAg^{-1}=A^{-1}$ or $gAg^{-1}=-A^{-1}$ for some $g\in \mathrm{SL}(2,\mathbb{H})$. 
	We classified reversible elements in $\mathrm{SL}(2,\mathbb{H})$ in Section~\ref{sec-rev}. 
	The following lemma classifies all elements $A\in \mathrm{SL}(2,\mathbb{H})$ such that $A$ is conjugate to $-A^{-1}$.
	
	\begin{lemma}\label{lem-rev-PSL-2}
		An element $A\in \mathrm{SL}(2,\mathbb{H})$ satisfies $gAg^{-1}=-A^{-1}$ for some $g\in \mathrm{SL}(2,\mathbb{H})$ if and only if it is conjugate in $\mathrm{SL}(2,\mathbb{H})$ to one of the following matrices:
		\begin{enumerate}[label={\normalfont(\arabic*)}]
			\item\label{rev-type-1-PSL}
			$\begin{pmatrix}
				e^{\ib \theta} & 0\\
				0 & -e^{\ib \theta}
			\end{pmatrix}$, where $\theta\in [0,\pi]$;
			
			\item\label{rev-type-2-PSL}
			$\begin{pmatrix}
				r e^{\ib \theta} & 0\\
				0 & -r^{-1} e^{\ib \theta}
			\end{pmatrix}$, where $r\in \R$, $r>0$, $r\neq 1$, and $\theta\in [0,\pi]$;
			
			\item\label{rev-type-3-PSL}
			$\begin{pmatrix}
				\ib & 1\\
				0 & \ib
			\end{pmatrix}$.
		\end{enumerate}
	\end{lemma}
	
	\begin{proof}
		For $A$ of types~\ref{rev-type-1-PSL}--\ref{rev-type-3-PSL}, consider the following matrices $g\in \mathrm{SL}(2,\mathbb{H})$, respectively:
		\[
		\begin{array}{lll}
			\text{(i)}~
			g := \begin{pmatrix} 0 & \jb \\ \jb & 0 \end{pmatrix}, & \quad
			\text{(ii)}~
			g := \begin{pmatrix} 0 & \jb \\ \jb & 0 \end{pmatrix}, & \quad
			\text{(iii)}~
			g := \begin{pmatrix} \ib & 0 \\ 0 & -\ib \end{pmatrix}.
		\end{array}
		\]
		In each case, one checks that $gAg^{-1}=-A^{-1}$. 
		Moreover, $g$ is a skew-involution in each case.
		
		Conversely, recall that $\jb z=\bar z\,\jb$ for all $z\in \C$. 
		For the unique complex representative $\lambda\in \C$ of an eigenvalue class of $A$, the numbers $\lambda$ and $\lambda^{-1}$ are conjugate if and only if $\lambda=\pm \ib$. 
		Also, $A$ is conjugate to $-A^{-1}$ if and only if $A$ and $-A^{-1}$ have the same Jordan form. 
		Therefore, if $A$ is conjugate to $-A^{-1}$ in $\mathrm{SL}(2,\mathbb{H})$, then Lemma~\ref{lem-conj} implies that $A$ is conjugate to one of the matrices listed in~\ref{rev-type-1-PSL}--\ref{rev-type-3-PSL}.
	\end{proof}
	
	\begin{proposition}\label{prop-rev-PSL-2-skew-inv-prod}
		Every element $A\in \mathrm{SL}(2,\mathbb{H})$ that is conjugate to $-A^{-1}$ can be written as a product of an involution and a skew-involution in $\mathrm{SL}(2,\mathbb{H})$.
	\end{proposition}
	
	\begin{proof}
		By Lemma~\ref{lem-rev-PSL-2}, we may assume that there exists a skew-involution $g\in \mathrm{SL}(2,\mathbb{H})$ such that $gAg^{-1}=-A^{-1}$. 
		Then
		\[
		A=\big(-g^{-1}A^{-1}\big)\,g,
		\]
		and
		\[
		\big(-g^{-1}A^{-1}\big)^2
		=(g^{-1}A^{-1}g^{-1})A^{-1}
		=A\,A^{-1}
		=\mathrm{I}_2.
		\]
		Hence $-g^{-1}A^{-1}$ is an involution, and the result follows.
	\end{proof}
	
	The following result shows that every reversible element in $\mathrm{PSL}(2,\mathbb{H})$ is strongly reversible.
	
	\begin{proposition}[cf.~{\cite{LFS}}]\label{prop-equiv-rev-str}
		An element $[A] \in \mathrm{PSL}(2,\H)$ is reversible if and only if it is strongly reversible.
	\end{proposition}

	\begin{proof}
		The implication $\Longleftarrow$ is immediate, so we prove the $\Longrightarrow$ direction. 
		Every involution and skew-involution in $\mathrm{SL}(2,\mathbb{H})$ projects to an involution in $\mathrm{PSL}(2,\mathbb{H})$. 
		Thus, by Propositions~\ref{prop-rev-prod-skew-inv} and~\ref{prop-rev-PSL-2-skew-inv-prod}, every reversible element of $\mathrm{PSL}(2,\mathbb{H})$ is a product of two involutions in $\mathrm{PSL}(2,\mathbb{H})$. 
		Hence every reversible element of $\mathrm{PSL}(2,\mathbb{H})$ is strongly reversible.
	\end{proof}
	
	We now prove our main result.
	
	\subsection{Proof of Theorem~\ref{thm-rev-alg-char}}
	
	Let $A\in \mathrm{SL}(2,\mathbb{H})$ be a lift of $[A]\in \mathrm{PSL}(2,\mathbb{H})$. 
	By~\eqref{eq-char-GL}, the characteristic polynomial of $A$ is
	\[
	\chi_{\H}(A)=x^4-c_3x^3+c_2x^2-c_1x+1,
	\]
	where $c_1,c_2,c_3\in \R$, and by~\eqref{eq-char-GL-relation} we have
	\[
	c_3=2\delta_A,\qquad c_2=\gamma_A,\qquad c_1=2\beta_A.
	\]
	Since $c_3=2\delta_A$ and $c_1=2\beta_A$, the condition $c_1=\pm c_3$ is equivalent to $\beta_A^2=\delta_A^2$. 
	Recall that $[A]\in \mathrm{PSL}(2,\mathbb{H})$ is reversible if and only if $A$ is conjugate to $\pm A^{-1}$ in $\mathrm{SL}(2,\mathbb{H})$. 
	Therefore, using~\eqref{eq-char-inverse}, if $[A]$ is reversible, then $c_1=\pm c_3$, which is equivalent to $\beta_A^2=\delta_A^2$.
	
	Conversely, assume that $\beta_A^2=\delta_A^2$, equivalently $c_1=\pm c_3$. 
	Suppose for contradiction that $[A]$ is not reversible in $\mathrm{PSL}(2,\H)$. 
	By Lemma~\ref{lem-conj} and Remark~\ref{rem-non-rev}, we have $c_1\neq c_3$ if and only if $A$ is not reversible in $\mathrm{SL}(2,\H)$, and hence $A$ is conjugate to
	\begin{equation}\label{eq-1}
		\begin{pmatrix}
			r e^{\ib \theta} & 0\\
			0 & r^{-1} e^{\ib \phi}
		\end{pmatrix},
	\end{equation}
	where $r\in \R$, $r>0$, and $\theta,\phi\in [0,\pi]$ satisfy $r\neq 1$ and $\theta\neq \phi$. 
	Moreover, by Lemmas~\ref{lem-conj} and~\ref{lem-rev-PSL-2}, we have $c_1\neq -c_3$ if and only if $A$ is not conjugate to $-A^{-1}$ in $\mathrm{SL}(2,\H)$. 
	Since $[A]$ is not reversible in $\mathrm{PSL}(2,\H)$, we obtain $c_1\neq \pm c_3$, contradicting the assumption. 
	Therefore, if $\beta_A^2=\delta_A^2$, then $[A]\in \mathrm{PSL}(2,\mathbb{H})$ is reversible.
	\qed
	
	The following remark follows from the proof of Theorem~\ref{thm-rev-alg-char}.
	
	\begin{remark}
		If $A$ is the matrix in~\eqref{eq-1} with $\theta=\pi-\phi$ (equivalently, $c_1=-c_3$), then Lemma~\ref{lem-rev-PSL-2}(2) implies that $[A]$ is reversible in $\mathrm{PSL}(2,\H)$. 
		Therefore, by Theorem~\ref{thm-rev-alg-char}, the following are equivalent:
		\begin{enumerate}[label={\normalfont(\roman*)}]
			\item $[A]$ is not reversible in $\mathrm{PSL}(2,\H)$;
			\item $c_1\neq \pm c_3$;
			\item $A$ is conjugate in $\mathrm{SL}(2,\H)$ to
			\[
			\begin{pmatrix}
				r e^{\ib \theta} & 0\\
				0 & r^{-1} e^{\ib \phi}
			\end{pmatrix},
			\]
			where $r\in \R$, $r>0$, $r\neq 1$, $\theta,\phi\in [0,\pi]$, $\theta\neq \phi$, and $\theta\neq \pi-\phi$.\qed
		\end{enumerate}
	\end{remark}
	
	\subsection{Reversing symmetries and involutions in $\mathrm{PSL}(2,\mathbb{H})$}\label{sec-inv-skew}
	
	In Lemma~\ref{lem-reverser}, we classified the reversing symmetries of reversible elements in $\mathrm{SL}(2,\mathbb{H})$. 
	Now, for $A\in \mathrm{SL}(2,\mathbb{H})$, define
	\begin{equation}\label{eq-def-S-set}
		\mathcal{S}_{\mathrm{SL}(2, \, \mathbb{H})}(A)
		:=
		\left\{
		g\in \mathrm{SL}(2,\mathbb{H})
		\ \middle|\
		gAg^{-1}=-A^{-1}
		\right\}.
	\end{equation}
	If $\mathcal{S}_{\mathrm{SL}(2, \, \mathbb{H})}(A)\neq \emptyset$, then $A$ is not reversible in $\mathrm{SL}(2,\mathbb{H})$, but $[A]$ is reversible in $\mathrm{PSL}(2,\mathbb{H})$. 
	The following lemma describes $\mathcal{S}_{\mathrm{SL}(2, \, \mathbb{H})}(A)$ when it is nonempty. 
	Together with Lemma~\ref{lem-reverser}, it yields a complete description of reversing symmetries in $\mathrm{PSL}(2,\mathbb{H})$.
	
	\begin{lemma}\label{lem-reverser-PSL}
		Let $A\in \mathrm{SL}(2,\mathbb{H})$ be one of the matrices in~\ref{rev-type-1-PSL}--\ref{rev-type-3-PSL} of Lemma~\ref{lem-rev-PSL-2}, and assume that $\mathcal{S}_{\mathrm{SL}(2, \, \mathbb{H})}(A)\neq \emptyset$. 
		Then $\mathcal{S}_{\mathrm{SL}(2, \, \mathbb{H})}(A)$ is given as follows in each case:
		\begin{enumerate}[label={\normalfont(\arabic*)}]
			\item If $\theta\neq \pi/2$, then
			$
			\mathcal{S}_{\mathrm{SL}(2, \, \mathbb{H})}(A)
			=
			\left\{
			\begin{pmatrix}
				0 & a\jb\\
				b\jb & 0
			\end{pmatrix}
			\ \middle|\
			a,b\in \mathcal{Z}(e^{\ib \theta})
			\right\}.
			$
			If $\theta=\pi/2$, then
			$
			\mathcal{S}_{\mathrm{SL}(2, \, \mathbb{H})}(A)
			=
			\left\{
			\begin{pmatrix}
				a & b\\
				c & d
			\end{pmatrix}
			\ \middle|\
			a,d\in \C,\ \ b,c\in \C\jb
			\right\}.
			$
			
			\item
			$
			\mathcal{S}_{\mathrm{SL}(2, \, \mathbb{H})}(A)
			=
			\left\{
			\begin{pmatrix}
				0 & a\jb\\
				b\jb & 0
			\end{pmatrix}
			\ \middle|\
			a,b\in \mathcal{Z}(e^{\ib \theta})
			\right\}.
			$
			
			\item
			$
			\mathcal{S}_{\mathrm{SL}(2, \, \mathbb{H})}(A)
			=
			\left\{
			\begin{pmatrix}
				a & b\\
				0 & -a
			\end{pmatrix}
			\ \middle|\
			a,b\in \C
			\right\}.
			$
		\end{enumerate}
	\end{lemma}
	
	\begin{proof}
		We omit the proof, as it follows by arguments similar to those in the proof of Lemma~\ref{lem-reverser}.
	\end{proof}
	
	We conclude this section with the following characterization of involutions in $\mathrm{PSL}(2,\mathbb{H})$.
	
	\begin{lemma}\label{lem-inv-PSL}
		Let $[A]$ be an involution in $\mathrm{PSL}(2,\mathbb{H})$ such that $A\neq \pm \mathrm{I}_2$. 
		Then $A$ is conjugate in $\mathrm{SL}(2,\mathbb{H})$ to one of the following matrices:
		\[
		\begin{pmatrix}
			1 & 0\\
			0 & -1
		\end{pmatrix},
		\qquad
		\begin{pmatrix}
			0 & 1\\
			-1 & 0
		\end{pmatrix}.
		\]
	\end{lemma}
	
	\begin{proof}
		Since $[A]\in \mathrm{PSL}(2,\mathbb{H})$ is an involution, we have $A^2=\pm \mathrm{I}_2$. 
		If $A^2=\mathrm{I}_2$, then $A$ is strongly reversible in $\mathrm{SL}(2,\mathbb{H})$. 
		Hence Lemma~\ref{lem-classification-str-rev-SL} implies that $A$ is conjugate to
		$\begin{pmatrix} 1 & 0\\ 0 & -1 \end{pmatrix}$. 
		If $A^2=-\mathrm{I}_2$, then $A=-A^{-1}$, and Lemma~\ref{lem-rev-PSL-2} implies that $A$ is conjugate to
		$\begin{pmatrix} 0 & 1\\ -1 & 0 \end{pmatrix}$ in $\mathrm{SL}(2,\mathbb{H})$.
	\end{proof}

	\bigskip

	\textbf{Acknowledgements.} The authors thank the anonymous referee for a careful reading of the manuscript and for valuable suggestions that improved its presentation. 
	
	Gongopadhyay acknowledges partial support from ANRF grants CRG/2022/003680 and ANRF/ARGM/2025/000122/MTR. Lohan acknowledges support from postdoctoral fellowships at IIT Kanpur and ISI Delhi. Mukherjee acknowledges partial support from NBHM grant Ref. No. 02011/35/2023/NBHM(R.P)/R\&DII/168877.

\end{document}